\documentclass[12pt]{article}
\usepackage{amsthm}
\usepackage{amsfonts}
\usepackage{amsmath, amssymb}
\usepackage{asymptote}

\newtheorem{theorem}{Theorem}
\newtheorem{lemma}[theorem]{Lemma}
\newtheorem{observation}[theorem]{Observation}

\newcommand{\cref}[1]{(\ref{#1})}
\newcommand{\fec}{{\cal C}_4}
\newcommand{\pec}{{\cal C}_5}
\newcommand{\subp}{{\cal P}}

\title{Graphs avoiding immersion of $K_{3,3}$}
\author{Zden\v{e}k Dvo\v{r}\'ak\thanks{Charles University, Prague, Czech Republic. E-mail: \texttt{rakdver@kam.mff.cuni.cz}.
Supported by the Neuron Foundation for Support of Science under Neuron Impuls programme.}\and
Michal Hru\v{s}ka\thanks{Charles University, Prague, Czech Republic.}}

\begin{document}
\maketitle

\begin{abstract}
DeVos and Malekian~\cite{devos2018structure} gave a structural description of graphs avoiding an immersion of $K_{3,3}$,
showing that all such graphs are composed over small edge-cuts from graphs
with at most 8 vertices and from $3$-regular planar graphs.  We provide another proof of this fact, simpler in some aspects.
\end{abstract}

\section{Introduction}

Throughout the paper, graphs are undirected and may contain parallel edges, but no loops.
An \emph{immersion} $\beta$ of a graph $H$ in $G$ assigns to vertices $v\in V(H)$ distinct
vertices $\beta(v)\in V(G)$, and to edges $e=uv\in E(H)$ pairwise edge-disjoint paths $\beta(e)$
from $\beta(u)$ to $\beta(v)$ in $G$.  The vertices $\beta(V(H))$ of $G$ are called the \emph{branch vertices}
of the immersion.

The notion of immersion is related to other graph inclusion notions such as \emph{topological minors}
(where the paths are required to be vertex-disjoint) and \emph{minors} (where vertices of $H$ are represented
by pairwise vertex-disjoint connected subgraphs of $G$).  The structural theory of graphs avoiding a fixed
graph as a minor~\cite{robertson2003graph} or as a topological minor~\cite{gmarx} has been a subject of a lot of
research, in part motivated by numerous theoretical and algorithmic applications.  In particular, for many small
graphs $H$, an exact characterization of graphs avoiding $H$ as a minor is known.  Relevantly for the present
work, a graph avoids $K_{3,3}$ as a minor (or a topological minor) if and only if it can be obtained from planar graphs and copies of $K_5$
by gluing on cliques of size at most two~\cite{wagner}.

The theory of classes that avoid a fixed graph as an immersion is somewhat less developed.  While there are
results on their general structure~\cite{wollims}, not much is known for particular graphs.
Giannopoulou et al.~\cite{giannopoulou2015forbidding} gave a partial result for $K_{3,3}$: They prove
that a sufficiently edge-connected graph avoiding $K_{3,3}$ as an immersion is either planar and $3$-regular,
or has small treewidth.  However, while planar $3$-regular graphs can easily be seen to never contain an immersion of $K_{3,3}$,
treewidth is not a natural parameter to consider in the setting of immersions---every graph can be immersed in some simple graph
of treewidth two.

In this paper, we present a more precise characterization: we show that a graph avoids $K_{3,3}$ as an immersion if and only if
it can be obtained from graphs with at most $8$ vertices without an immersion of $K_{3,3}$ and from planar $3$-regular graphs
via well-described operations (compositions over small edge-cuts).  Thus, the characterization is exact up to enumeration
of the graphs without immersion of $K_{3,3}$ with at most $8$ vertices.

Shortly after finishing this writeup, we learned that an essentially identical result was obtained a few months before
by DeVos and Malekian~\cite{devos2018structure}.  Compared to the present paper, they also include the exact description of the small
obstructions, obtained by computer-assisted enumeration\footnote{We only use the computer enumeration to show that there are no
obstructions with $9$ vertices, and to prove a specific property of the obstructions with $6$ to $8$ vertices, see Lemmas~\ref{lemma-no9} and \ref{lemma-no678}.
However, the programs we use could be extended to provide an explicit list of obstructions, which we initially planned to do but decided
not to once we learned the result has already been obtained in~\cite{devos2018structure}.}.  On the other hand, their overall argument is somewhat more involved
than ours and uses another non-trivial structural result on graphs avoiding an immersion of the wheel with four spokes~\cite{devos2018structurew4}.
Furthermore, as the computer-assisted enumeration is used in both proofs, we believe an independent confirmation of the result may be of interest.

\section{Preliminaries}

Let $G$ be a graph.
Let $e_1$ and $e_2$ be non-parallel edges incident with the same vertex $v$ of $G$, and let $u_1$ and $u_2$ be the
ends of these edges other than $v$.  By \emph{splitting off the edges $e_1$ and $e_2$}, we mean deleting $e_1$ and $e_2$
and adding a new edge between $u_1$ and $u_2$.  Note that if a graph $H$ immerses in the resulting graph, then $H$ also immerses in $G$.

In a connected graph $G$, a \emph{separation}
is a pair $(A,B)$ of non-empty disjoint sets $A,B\subseteq V(G)$ such that $V(G)=A\cup B$,
and a \emph{cut} associated with the separation is the set $\partial(A,B)$ of edges of $G$
with one end in $A$ and the other end in $B$ (or $\partial_G(A,B)$ when the graph $G$ is not clear from the context);
we will also write $\partial A$ or $\partial B$ to denote this set.
The \emph{order} of the cut or separation
is the number of such edges.  The sets $A$ and $B$ are the \emph{sides} of the separation
or cut (when $G$ is connected, the cut determines its sides).
We often use the following obvious fact.

\begin{observation}\label{obs-cs}
If $(A,B)$ is a separation of $H$ of order $k$ and $\beta$ is an immersion of $H$ in $G$,
then $G$ contains $k$ pairwise edge-disjoint paths from $\beta(A)$ to $\beta(B)$;
and in particular, $G$ does not contain a separation $(C,D)$ of order less than $k$ with $\beta(A)\subseteq C$
and $\beta(B)\subseteq D$.
\end{observation}

A graph is \emph{$k$-edge-connected} if it has no cuts of order less than $k$.
It is \emph{internally $4$-edge-connected} if it is $3$-edge-connected and every cut of order $3$
has a side consisting of a single vertex.  A graph is \emph{weakly $5$-edge-connected} if it is internally
$4$-edge-connected and every cut of order $4$ has a side consisting of at most two vertices.
We say that a side $A$ of a cut of order $k$ is \emph{acceptable} if $k=3$ and $|A|=1$, or $k=4$ and $|A|\le 2$, or $k\ge 5$;
hence, $G$ is weakly $5$-edge-connected if and only if every cut has an acceptable side.

Let $H_1$ and $H_2$ be graphs with at least three vertices and let $v_1\in V(H_1)$ and $v_2\in V(H_2)$ be vertices of the same
degree $d$.  A graph $G$ is obtained from graphs $H_1$ and $H_2$ by a \emph{join on $v_1$ and $v_2$}
if for some bijection $\pi$ between the edges incident with $v_1$ and $v_2$, $G$ is created by removing the
vertices $v_1$ and $v_2$, and for each edge $e$ incident with $v_1$, adding the edge between the vertices
incident with $e$ and $\pi(e)$ distinct from $v_1$ and $v_2$.  Note that $(V(H_1-v_1),V(H_2-v_2))$ is a separation in $G$
of order $d$, which we say to be \emph{associated} with the join.  We need the following observation.

\begin{lemma}\label{lemma-3join}
Suppose $G$ is a join of $H_1$ and $H_2$ on vertices $v_1$ and $v_2$ of degree at most three.  If $G$ contains
$K_{3,3}$ as an immersion, then $H_1$ or $H_2$ contains $K_{3,3}$ as an immersion.  If $H_1-v_1$ and $H_2-v_2$ are connected,
then the converse holds as well.
\end{lemma}
\begin{proof}
Let $(A,B)=(V(H_1-v_1),V(H_2-v_2))$ be the separation of $G$ of order at most three associated with the join.

Suppose $G$ contains $K_{3,3}$ as an immersion $\beta$.
Since $K_{3,3}$ is internally $4$-edge-connected, Observation~\ref{obs-cs} implies that at most one of $A$ and $B$
contains more than one branch vertex of $\beta$.  If say $B$ does not contain any branch vertex, then at most one path of $\beta$
passes through $B$.  If $B$ contains exactly one branch vertex $v$, then only the paths of $\beta$ ending in $v$ intersect $B$.
In either case, contracting $B$ to a single vertex transforms $G$ to $H_1$ and $\beta$ to an immersion of $K_{3,3}$ in $H_2$.

Suppose now that say $H_1$ contains $K_{3,3}$ as an immersion and $H_2-v_2$ is connected.  Since $\deg(v_2)\le 3$,
there exists a vertex $w\in B$ joined to $v_2$ in $H_2$ by $\deg(v_2)$ pairwise edge-disjoint paths.  By considering $w$ and the paths
as a replacement for $v_1$ and the incident edges, we conclude that $H_1$ is immersed in $G$.  Since the relation of immersion is transitive,
it follows that $K_{3,3}$ is immersed in $G$.
\end{proof}

\section{Graphs of small edge-connectivity}

Lemma~\ref{lemma-3join} allows us to restrict our attention only to internally $4$-edge-connected graphs.
Let $\fec$ be the class of internally $4$-edge-connected graphs without immersion of $K_{3,3}$.

\begin{lemma}\label{lemma-4con}
A graph does not contain $K_{3,3}$ as an immersion if and only if $G$ is obtained from graphs in $\fec$
by repeated application of the following operations:
\begin{itemize}
\item joins on vertices of degree at most three,
\item disjoint unions,
\item subdividing edges, and
\item adding pendant vertices joined by single or double edges.
\end{itemize}
\end{lemma}
\begin{proof}
By Lemma~\ref{lemma-3join}, joins on vertices of degree at most three cannot create an immersion of $K_{3,3}$,
and clearly neither can disjoint unions, edge subdivisions, or adding pendant vertices of degree at most two.  Hence, the graphs arising as described
do not contain $K_{3,3}$ as an immersion.

Conversely, we will prove by induction on the number of vertices that every $K_{3,3}$-immersion-free graph can be created as described.
Suppose $G$ is $K_{3,3}$-immersion-free and the claim holds for all graphs with fewer than $|V(G)|$ vertices.
If $G$ is disconnected, then it is a disjoint union of its components, which are $K_{3,3}$-immersion-free; hence, the claim follows by
the induction hypothesis.  If $G$ contains a vertex of degree one or two, then $G$ is obtained from a smaller $K_{3,3}$-immersion-free
graph by either subdividing an edge or adding a pendant vertex joined by a single or double edge, and again the claim follows by the induction
hypothesis.  Therefore, we can assume that $G$ is connected and has minimum degree at least three.
The claim is trivial if $G$ is internally $4$-edge-connected; hence, we can assume $G$ has a separation $(A,B)$
of order $k\le 3$, such that $|A|,|B|\ge 2$ if $k=3$.  If $k\le 2$, note that $|A|,|B|\ge 2$ as well, since $\delta(G)\ge 3$.
Consider such a separation with $k$ minimum; then $G[A]$ and $G[B]$ are connected.  In this case, $G$ is a join of graphs $H_1$ and $H_2$
on vertices of degree $k$ with associated separation $(A,B)$, and by Lemma~\ref{lemma-3join}, the graphs $H_1$ and $H_2$
are $K_{3,3}$-immersion-free.  Since $|A|,|B|\ge 2$, we have $|V(H_1)|,|V(H_2)|<|V(G)|$, and thus the claim follows by the induction hypothesis.
\end{proof}

Hence, we only need to describe the structure of graphs in $\fec$.

\section{Non-planar graphs}

The case of non-planar graphs is quite simple, thanks to Kuratowski theorem.
If $H$ is a subgraph of $G$, then an \emph{immersion $H$-bridge} in $G$ is a connected component of
the graph $G-E(H)$ with at least two vertices; hence, distinct vertices of $G$ are joined by a path
edge-disjoint from $H$ if and only if they belong to the same immersion $H$-bridge.  Note
also that immersion $H$-bridges are pairwise vertex-disjoint.
The \emph{attachments} of an immersion $H$-bridge $K$ are the vertices of $V(H)\cap V(K)$.

\begin{lemma}\label{lemma-nonpl}
If $G$ is a $3$-edge-connected non-planar graph with at least $6$ vertices,
then $G$ contains an immersion of $K_{3,3}$.
\end{lemma}
\begin{proof}
By Kuratowski theorem, $G$ contains a subdivision of $K_5$ or $K_{3,3}$.  In the latter case
the subdivision gives also an immersion, hence assume the former; let $H$ be a subgraph of $G$ isomorphic
to a subdivision of $K_5$.  Suppose first that $H$ contains a vertex of degree two (i.e., at least one of the edges of $K_5$
is subdivided), and let $P$ be a maximal path of vertices of degree two in $G$.  Since $G$ is $3$-edge-connected, $G$ contains
an immersion $H$-bridge $K$ intersecting both $V(P)$ and $V(H)\setminus V(P)$.
A straightforward case analysis shows that $H$ together with a path in $K$ (and thus also $G$) contains $K_{3,3}$ as an immersion.

Therefore, we can assume $H=K_5$ is a subgraph of $G$.  Since $|V(G)|>5$, there exists a vertex $v\in V(G)\setminus V(H)$,
and since $G$ is $3$-edge-connected, it contains three pairwise edge-disjoint paths from $v$ to $V(H)$.  Again, a straightforward
case analysis shows that $H$ together with these paths (and thus also $G$) contains $K_{3,3}$ as an immersion.
\end{proof}

In particular, all non-planar graphs in $\fec$ have at most five vertices (and since every graph immersing $K_{3,3}$
has at least six vertices, all internally $4$-edge-connected graphs with at most five vertices belong to $\fec$).

\section{Internally $4$-edge-connected graphs}

It will be convenient to improve the connectivity a bit.
Let $\pec$ be the class of internally $5$-edge-connected graphs without immersion of $K_{3,3}$.
In this section, we describe how the graphs in $\fec$ are obtained from the graphs in $\pec$.

Let $\subp$ denote the class of internally $4$-edge-connected $3$-regular planar graphs.
Note that if $H$ is immersed in a $3$-regular graph $G$, then a subdivision of $H$ is a subgraph of $G$;
consequently, all graphs immersed in the graphs from $\subp$ are planar, and thus $K_{3,3}$ is not immersed in
any graph from $\subp$.

Let $(A,B)$ be a separation in a graph $G$ with $\partial(A,B)=\{e_1,e_2,e_3,e_4\}$.
We say that \emph{$B$ immerses an $\{e_1,e_2,e_3,e_4\}$-vertex} if $G$ contains four pairwise edge-disjoint paths
starting in $e_1$, \ldots, $e_4$ and ending in the same vertex in $B$.
We say that \emph{$B$ immerses an $(e_1,e_2,e_3,e_4)$-span}, if there exist distinct vertices $u,v\in B$
and pairwise edge-disjoint paths in $G$ joining $e_1$ and $e_2$ to $u$, $e_3$ and $e_4$ to $v$, and $u$ to $v$.

A graph $G$ is obtained from graphs $F_1$ and $F_2$ with at least five vertices by a \emph{special $4$-join} if for $1\le i\le 2$,
$F_i$ contains a vertex $v_i$ of degree three joined by three edges to a vertex $w_i$ of degree seven, and
$G$ is obtained from $F_1-v_1$ and $F_2-v_2$ by a join on $w_1$ and $w_2$.  The separation $(V(F_1)\setminus \{v_1,w_1\},V(F_2)\setminus\{v_2,w_2\})$
of $G$ is \emph{associated} with the special $4$-join.

A graph $G$ is obtained from a graph $H$ by \emph{pinching a double edge} if $H$ contains exactly
two edges between some vertices $u$ and $v$, and $G$ is created by removing these two edges,
adding a new vertex $w$ and adding double edges between $w$ and $u$, and between $w$ and $v$.

\begin{lemma}\label{lemma-4conn}
A graph belongs to $\fec$ if and only if either it belongs to $\subp$ or it is obtained from graphs in $\pec\setminus\subp$
by repeated application of the following operations:
\begin{itemize}
\item pinching a double edge between vertices of degree four, and
\item special $4$-joins.
\end{itemize}
\end{lemma}
\begin{proof}
Suppose that $G\in \fec$ and $G'$ is obtained from $G$ by pinching a double edge between vertices $u$ and $v$ of degree
four, introducing a new vertex $w$.
If $G'$ were not internally $4$-edge-connected, it would contain a separation $(A,B)$ of order three with say $u,w\in A$ and $|B|\ge 2$.
Since $G$ is internally $4$-edge-connected and its separation $(A\setminus\{w\},B)$ also has order three, we have $|A\setminus\{w\}|=1$,
and thus $A=\{u,w\}$.  However, since $(A,B)$ has order three and $wv$ is a double edge, this implies $\deg(u)=3$, which is a contradiction.
Hence, $G'$ is internally $4$-edge-connected.  Suppose now $G'$ contains $K_{3,3}$ as an immersion.  Since this immersion does not appear in $G$,
$w$ is a branch vertex of this immersion.  Since any cut separating three vertices of $K_{3,3}$ has size at least $5$ but only four edges
leave $\{u,w,v\}$ since $\deg(u)=\deg(v)=4$ we can by symmetry assume that $u$ is not a branch vertex.
But then we can modify the immersion so that $u$ is the branch vertex instead of $w$, and after suppressing $w$ we would obtain an immersion of $K_{3,3}$
in $G$.  This is a contradiction, and thus $G'\in \fec$.

Suppose that $G_1,G_2\in \fec$ and $G'$ is obtained from them by a special $4$-join (and in particular $|V(G_1)|,|V(G_2)|\ge 5$), with an associated separation $(A,B)$ of order four.
Let us consider the case that $G'$ is not internally $4$-edge-connected due to a separation $(C,D)$ of order at most three.  If $A\cap C=\emptyset$,
then $(C,V(G_2)\setminus C)$ is a separation of $G_2$ contradicting the assumption that $G_2$ is internally $4$-edge-connected.
Hence, $A\cap C\neq\emptyset$, and symmetrically $A\cap D,B\cap C,B\cap D\neq\emptyset$.  Since $|V(G_1)|\ge 5$,
we have $|A|\ge 3$, and since $G_1$ is internally $4$-edge-connected, we have $|\partial(A\cap C)|+|\partial(A\cap D)|\ge 7$.
The number of edges of $G'$ between $A\cap C$ and $A\cap D$ therefore is at least $\lceil (|\partial(A\cap C)|+|\partial(A\cap D)|-|\partial(A,B)|)/2\rceil\ge 2$,
and symmetric argument shows that there are at least two edges between $B\cap C$ and $B\cap D$.  This is a contradiction, since $|\partial(C,D)|\le 3$.
Hence, $G'$ is internally $4$-edge-connected.
If $G'$ contains an immersion of $K_{3,3}$, then $B$ cannot contain exactly three branch vertices, since any cut separating three vertices of $K_{3,3}$ has size at least $5$.
Hence, by symmetry we can assume that at most two branch vertices are contained in $B$.  But then the part of the $K_{3,3}$ immersed in $B$
can also be immersed in two vertices joined by a triple edge, and thus $K_{3,3}$ would also be immersed in $G_1$, which is a contradiction.
We conclude that $G'\in \fec$.  Therefore, all graphs obtained from the described construction indeed belong to $\fec$.

Conversely, we will prove by induction on the number of vertices that every graph in $\fec$ can be obtained by the construction.
Consider a graph $G\in \fec$.  If $G$ is weakly $5$-edge-connected, then the claim is trivial, and thus we can assume that
$G$ contains a separation $(A,B)$ of order four with $|A|,|B|\ge 3$.  Consequently, $|V(G)|\ge 6$, and thus 
by Lemma~\ref{lemma-nonpl}, $G$ is planar; we will from now on work with some fixed plane drawing of $G$.  If $G$ is $3$-regular, then the claim is again trivial, and thus
we can assume $A$ contains a vertex $v$ of degree at least four.

Since $G$ is $3$-edge-connected, $G[B]$ is connected, and thus it is drawn inside a single face $f$ of $G[A]$.
Let $e_1$, $e_2$, $e_3$ and $e_4$ be the edges of $\partial(A,B)$ in the cyclic order according to the drawing of their
ends in the boundary walk of $f$.  Let $S=\{e_1,e_2,e_3,e_4\}$.  Since $G$ is internally $4$-edge-connected,
it contains pairwise edge-disjoint paths from $v$ to $S$, and thus $A$ immerses an $S$-vertex.

Consider now the walks $W_1$ between $e_1$ and $e_2$ and $W_2$ between $e_3$ and $e_4$ in the boundary of $f$.
Note that $W_1$ and $W_2$ are edge-disjoint: if they shared an edge $e$, then $\{e_1,e_4,e\}$ and $\{e_2,e_3,e\}$
would be cuts in $G$ of order three, and since $G$ is internally $4$-edge-connected, this would imply $|A|\le 2$,
which is a contradiction.
If $E(G[A])\cup S=E(W_1)\cup E(W_2)$, then (since $\delta(G)>2$) all vertices of $A$ belong to both $W_1$ and $W_2$ and $G[A]$ is
a path of double edges; hence, $G$ is obtained by pinching a double edge from the graph $G_0$ obtained from $G$ by suppressing the internal vertices of $W_1$,
the graph $G_0$ belongs to $\fec\setminus\subp$ as we argued before, and the claim follows by the induction hypothesis.
Therefore, we can assume there exists an immersion $(W_1\cup W_2)$-bridge $H$ in $G[A]$.
If $H$ were vertex-disjoint from $W_1$, then the two edges of $W_2$ entering and leaving the first and the last vertex
in that $W_2$ intersects $H$ would form a cut in $G$, contrary to the $3$-edge-connectivity of $G$.
By symmetry, we conclude that $H$ intersects both $W_1$ and $W_2$.  If $|V(H)\cap (V(W_1)\cup V(W_2))|\ge 2$, then $H$ contains distinct
vertices $v_1\in V(W_1)$ and $v_2\in V(W_2)$, and $W_1$, $W_2$ and a path between $v_1$ and $v_2$ in $H$ certify that
$A$ immerses an $(e_1,e_2,e_3,e_4)$-span. If $|V(H)\cap (V(W_1)\cup V(W_2))|=1$, then let $v_1$ be the vertex in this intersection
and $v_2$ another vertex of $H$.  Since $G$ is $3$-edge-connected, $H$ contains three edge-disjoint paths between $v_1$ and $v_2$, and
again, it is easy to see that $A$ immerses an $(e_1,e_2,e_3,e_4)$-span.  By symmetry, we can also assume $A$ immerses an $(e_2,e_3,e_4,e_1)$-span.

Since $A$ immerses an $S$-vertex, $G$ contains edge-disjoint paths $W_1$ between $e_1$ and $e_3$ and $W_2$ between $e_2$ and $e_4$;
choose such paths with $E(W_1\cup W_2)$ minimal.
Suppose that there exists an immersion $(W_1\cup W_2)$-bridge $H$ in $G[A]$.  If $H$ intersects both $W_1$ and $W_2$,
then arguing as in the previous paragraph, we conclude that $A$ immerses an $(e_1,e_3,e_2,e_4)$-span.  If $H$ intersects
say only $W_1$, then consider the minimal subpath $P$ of $W_1$ containing all the attachments of $H$.  Since $G$ is $3$-edge-connected,
$P$ intersects $W_2$ in a vertex $w$.  We can now route $W_1$ through $H$ to obtain a path $W'_1$ between $e_1$ and $e_3$ edge-disjoint from $W_2$
and such that $P$ is a part of an immersion $(W_1\cup W_2)$-bridge $H$ in $G[A]$ intersecting both $W'_1$ and $W_2$; and again, we conclude that
$A$ immerses an $(e_1,e_3,e_2,e_4)$-span.  Finally, suppose that there exists no immersion $(W_1\cup W_2)$-bridge in $G[A]$,
and thus $E(G[A])\cup S=E(W_1)\cup E(W_2)$.  If $G[A]$ is a path of double edges, then $G$ is obtained by pinching a double edge
and the claim holds by the induction hypothesis.  Otherwise, let $xy$ be an edge of $W_2-\{e_2,e_4\}$ not parallel to an edge of $W_1$,
let $P$ be the subpath of $W_1$ between $x$ and $y$, let $W'_1$ be obtained from $W_1$ by replacing $P$ by $xy$ and let $W'_2$
be obtained from $W_2$ by replacing $xy$ by $P$.  Then $W'_1$ and $W'_2$ are edge-disjoint walks from $e_1$ to $e_3$ and from $e_2$ to $e_4$;
but $W'_2$ visits all internal vertices of $P$ twice, and thus there exists a path $W''_2$ from $e_2$ to $e_4$ with $E(W''_2)\subsetneq E(W'_2)$,
contradicting the assumption $E(W_1\cup W_2)$ is minimal.

Therefore, we can assume $A$ immerses all $(e_{\pi(1)}, e_{\pi(2)}, e_{\pi(3)}, e_{\pi(4)})$-spans, where $\pi$ is any permutation of $\{1,\ldots,4\}$.
If all vertices of $B$ have degree three in $G$, then since $|B|\ge 3$ and $G$ is internally $4$-edge-connected,
we conclude that $G[B]$ is $2$-edge-connected, and thus also $2$-connected.  Therefore, the face of $G[B]$ in which $G[B]$ is drawn
is bounded by a cycle $K$, and the edges of $S$ attach to distinct vertices of $K$.
The cycle $K$ together with the $(e_1,e_3,e_2,e_4)$-span in $A$, this gives an immersion of $K_{3,3}$ in $G$, which is a contradiction.

Therefore, at least one vertex of $B$ has degree at least four, and thus by symmetry, we can also assume $B$ immerses an $S$-vertex
as well as all possible spans.  For $X\in \{A,B\}$, let $G_X$ be the graph obtained from $G[X]$ by adding vertices
$v_X$ and $w_X$, where $v_X$ is joined to
$w_X$ by three edges and $w_X$ is additionally adjacent to the ends of edges of $S$ (with multiplicity).  Note that $G_A$ does not
contain an immersion of $K_{3,3}$, as any such immersion would either contain no branch vertices in $\{v_A,w_A\}$ and pass at most two paths
through this set, or contain one branch vertex in $\{v_A,w_A\}$ and only pass the paths ending in this branch vertex through this set,
or contain two adjacent branch vertices and only pass the paths ending in these branch vertices through this set.
In the first two cases, this would imply that that $G$ contains an immersion of $K_{3,3}$, since $B$ immerses an $S$-vertex.
In the last case, this would imply that that $G$ contains an immersion of $K_{3,3}$, since $B$ immerses the corresponding span.
Symmetrically, $G_B$ does not contain an immersion of $K_{3,3}$.  Observe furthermore that $G_A$ and $G_B$ are internally $4$-edge-connected,
and thus they belong to $\fec\setminus\subp$.  Since $G$ is a special $4$-join of $G_A$ and $G_B$, the claim follows by the induction hypothesis.
\end{proof}

Hence, we only need to describe the structure of graphs in $\pec\setminus \subp$.  In the following section, we prove
these graphs have at most $8$ vertices, thus finishing the argument.

\section{Weakly $5$-edge-connected planar graphs}

Let us give some observations on graphs in $\pec$; by Lemma~\ref{lemma-nonpl}, all such graphs with at least $6$ vertices
are planar, and we consider them with some fixed drawing in the plane.  We focus on graphs in $\pec$ with the smallest possible number of edges
for a given number of vertices: let $\underline{\pec}$ denote the set of graphs $G\in \pec\setminus\subp$ such that no graph with $|V(G)|$ vertices and less than $|E(G)|$ edges
belongs to $\pec\setminus\subp$; i.e., every weakly $5$-edge connected graph with $|V(G)|$ vertices and less than $|E(G)|$ edges
either is planar $3$-regular, or contains an immersion of $K_{3,3}$.

\begin{lemma}\label{lemma-2con}
All graphs $G\in\underline{\pec}$ with at least $6$ vertices are $2$-connected.
\end{lemma}
\begin{proof}
Suppose for a contradiction $v$ is a cut-vertex in $G$, and thus there exist non-empty disjoint sets $C,D\subset V(G)\setminus \{v\}$
such that $V(G)=C\cup D\cup \{v\}$ and no edge of $G$ has one end in $C$ and the other end in $D$.
Since $G$ is $3$-edge-connected, $v$ is joined by at least three edges to each of $C$ and $D$, and thus $\deg(v)\ge 6$.
Let $e_C=vx$ and $e_D=vy$ be arbitrary edges joining $v$ to $C$ and $D$, respectively, and let $G'$ be the graph obtained from $G$ by splitting off $e_C$ and $e_D$.
Clearly $G'$ does not contain $K_{3,3}$ as an immersion, and since $\deg_{G'}(v)\ge 4$, it is not $3$-regular.  Since $G\in \underline{\pec}$, $V(G')=V(G)$ and $|E(G')|<|E(G)|$,
we conclude $G'$ is not weakly $5$-edge-connected.  Hence, $G'$ has a separation $(A,B)$ without an acceptable side, where $v\in A$.
We have $|\partial_G(A,B)|=|\partial_{G'}(A,B)|-[|\{x,y\}\cap B|=1]+|\{x,y\}\cap B|$.  If $\{x,y\}\not\subseteq B$, this would imply
$|\partial_G(A,B)|=|\partial_{G'}(A,B)|$ and the separation $(A,B)$ would not have an acceptable side in $G$, which is a contradiction.

Therefore, $\{x,y\}\subseteq B$ and $|\partial_G(A,B)|=|\partial_{G'}(A,B)|+2$.  However, then $B\cap C\neq \emptyset$ and $B\cap D\neq\emptyset$, and since $G[B]\subseteq G[C]\cup G[D]$,
we conclude $G[B]$ is not connected.  Since $G$ is $3$-edge-connected, it follows that $|\partial_G(B)|\ge |B|+4$,
and thus $|\partial_{G'}(A,B)|=|\partial_G(B)|-2\ge |B|+2$.  Consequently, $B$ is an acceptable side of $(A,B)$ in $G'$, which is a contradiction.
\end{proof}

The next lemma restricts neighborhoods of vertices of degree at least four in graphs from $\underline{\pec}$.

\begin{lemma}\label{lemma-3nbr}
Suppose $G\in\underline{\pec}$ has at least $6$ vertices.
Let $v$ be a vertex of $G$ of degree at least four such that all edges incident with $v$ have
multiplicity one and all vertices adjacent to $v$ have degree three.  Then $G$ is either the wheel with $5$ spokes
or the graph $K'_{2,4}$ obtained from $K_{2,4}$ by doubling the edges incident with one of the vertices of degree four.
\end{lemma}
\begin{proof}
By Lemma~\ref{lemma-2con}, the graph $G-v$ is connected.
Let $f$ be the face of $G-v$ whose interior contains $v$.
Let $e_1$, $e_2$, \ldots, $e_d$ be the edges incident with $v$ in the cyclic order according to the drawing of $G$ and $v_1$, \ldots, $v_d$
the other vertices incident with these edges.  

Suppose first that $G-v$ is not $2$-edge-connected.  Let $C_1$ and $C_2$ be vertex-disjoint maximal $2$-edge-connected
subgraphs of $G-v$ such that each of them is joined to the rest of $G-v$ by a single edge $g_1$ or $g_2$, respectively.
Let $P$ be a path in $G-v$ joining $g_1$ to $g_2$.  Since $G$ is weakly $5$-edge-connected and all edges incident with $v$ have
multiplicity one, at least four distinct neighbors of $v$ belong to each of $C_1$ and $C_2$.  Since $C_1$ and $C_2$ are $2$-edge-connected,
there exists an open tour containing four of the neighbors of $v$ in $C_1$, the path $P$ and four of the neighbors of $v$
in $C_2$.  However, such a graph together with the eight edges to $v$ immerses $K_{3,3}$, with three of the branch vertices contained
in $C_1$ and the other three in $C_2$.  This is a contradiction, and thus $G-v$ is $2$-edge-connected.

Let $W$ be the boundary walk of $f$.  For $i=1,\ldots, d$, let $W_i$ be the subwalk of $W$ between $v_i$ and $v_{i+1}$,
where $v_{d+1}=v_1$.  Since $G$ is $2$-connected, $W_1$, \ldots, $W_d$ are paths.
Since $G-v$ is $2$-edge-connected, the paths $W_1$,\ldots, $W_d$ are pairwise edge-disjoint.
Let $H$ be the subgraph of $G$ consisting of $W$ and the edges incident with $v$.
Since $G$ does not contain an immersion of $K_{3,3}$, we have $\deg(v)\le 5$.

Suppose that an immersion $H$-bridge $B$ attaches only to vertices of two consecutive paths $W_i$ and $W_{i+1}$. Let $f_1$ be the first edge and $f_2$ the last edge
of $W_i\cup W_{i+1}$ that is incident with a vertex of $B$, and let $T$ be the subtour of $W_i\cup W_{i+1}$ between these two edges (but excluding them).
Note that by planarity, no immersion $H$-bridge is incident both with a vertex of $T$ and a vertex outside of $T$.  Furthermore, no vertex of $T$
belongs to a path $W_j$ for $j\not\in \{i,i+1\}$, since $B$ does not attach to $W_j$.  It follows that $\{f_1,f_2,e_{i+1}\}$
is a cut in $G$.  But since $B$ contains at least one edge, both sides of this cut have size at least two, contradicting the assumption that $G$ is weakly $5$-edge-connected.
Similarly, we exclude the case that an immersion $H$-bridge attaches only to $W_i$ for some $i$.

Let us now consider the case that $v$ has degree $5$.  Note that for any $i$, $W_i$ and $W_{i+2}$ are vertex-disjoint and
no immersion $H$-bridge attaches to internal vertices of both of these paths, as otherwise $G$ would contain an immersion of $K_{3,3}$
(no immersion $H$-bridge attaches to their endvertices, since they have degree three).  Consequently, there are no immersion $H$-bridges and $G=H$.
Observe that $W_i$ and $W_{i+1}$ cannot share an internal vertex for any $i$, as otherwise $G$ would contain a cut of size three with both sides of size
at least two.  Therefore, $G=H$ and $G$ is the wheel with $5$ spokes.

Finally, suppose that $\deg(v)=4$.  Suppose that one of the paths $W_i$, say $W_1$, is a single edge.  But then the edges of $W_2$ and $W_4$ incident with $v_2$ and $v_1$,
respectively, together with $e_3$ and $e_4$, form a cut of size four with both sides of size at least three.
This is a contradiction with weak $5$-edge-connectivity of $G$, and hence each path $W_i$ has length at
least two.  If $|V(G)|=6$, this implies that $G=K'_{2,4}$, hence assume that $|V(G)|\ge 7$.  As we argued before, no immersion $H$-bridge
attaches only to vertices of $W_i\cup W_{i+1}$, for any $i$.  Suppose that $G\neq H$, and thus an immersion $H$-bridge $B$ attaches to vertices of say $W_1$ and $W_3$
(and possibly also to $W_2$ or $W_4$).  The attachments of $B$ are internal vertices of the paths, since all neighbors of $v$ have degree
three.  If $B$ has at least two attachments in $W_1\cup W_3$, then $G$ contains an immersion of $K_{3,3}$.  Therefore, the only attachment of $B$ in $W_1\cup W_3$
is a common internal vertex $w$ of $W_1$ and $W_3$.  Since $G$ is $2$-connected, $B$ has another attachment distinct from $w$, which belongs to $W_2\cup W_4$; but then $G$ contains
an immersion of $K_{3,3}$.

We conclude that there are no immersion $H$-bridges, and thus $G=H$.
Since each path $W_i$ has length at least two and $G$ is $3$-edge connected, each path $W_i$ shares an internal vertex with at least one of the other paths.
By planarity, it cannot be the case that $W_1$ intersects only $W_3$ and $W_2$ intersects only $W_4$; therefore, we can by symmetry assume that
$W_1$ and $W_2$ intersect in an internal vertex $w$.
Then $w$ belongs also to $W_3\cup W_4$, as otherwise $G$ would contain a cut of order three with both sides of size at least two.  By symmetry, we can assume that $w$ belongs to $W_3$.
Since every cut of order three has a side of size one, note that $w$ is the only common internal vertex in $W_1\cap W_2$ and in $W_2\cap W_3$, and
hence $w$ is joined to each of $v_2$ and $v_3$ by a double edge.
Since $G\neq K'_{2,4}$, $w$ does not belong to $W_4$.  However, then both sides of the cut consisting of the edges of $W_1$ and $W_3$ incident with $w$ (and not with $v_2$ and $v_3$)
together with $e_2$ and $e_3$ have size at least three.  This is a contradiction with the assumption that $G$ is weakly $5$-edge-connected.
\end{proof}

To ensure that we need to process only finitely many graphs in the computer-assisted enumeration part of the proof,
we need a bound on the multiplicity of the edges.

\begin{lemma}\label{lemma-maxmul}
No graph $G\in\underline{\pec}$ with at least $6$ vertices contains an edge of multiplicity greater than three.
\end{lemma}
\begin{proof}
Suppose for a contradiction vertices $u,v\in V(G)$ are joined by at least four edges,
and let $e$ be one of the edges.  Clearly $\max(\deg(u),\deg(v))\ge 6$, and in particular $G-e$ is not $3$-regular.
Since $G$ does not contain $K_{3,3}$ as an immersion, neither does $G-e$, and since $G\in \pec\setminus\subp$,
we conclude that $G-e$ is not weakly $5$-edge-connected.  Let $(A,B)$ be a separation of $G-e$ with no acceptable side.
Since $(A,B)$ has an acceptable side in $G$, we can assume $u\in A$ and $v\in B$.  Consequently, $4\ge |\partial_{G-e}(A,B)|=|\partial(A,B)|-1\ge 3$,
and since no side of $(A,B)$ is acceptable in $G-e$, we have $|A|,|B|\ge 2$.  By Lemma~\ref{lemma-2con}, neither $u$ nor $v$ is a cut-vertex in $G$,
and thus $|\partial_G(A,B)|=5$ (and thus $|A|,|B|\ge 3$), there are exactly four edges between $u$ and $v$,
and $\partial_G(A,B)$ contains an edge $e'=xy$ with $\{x,y\}\cap \{u,v\}=\emptyset$.

We claim that $G[A]$ contains an immersion $\beta$ of a triangle $rst$ with a double edge $rs$ such that $\beta(r)=u$ and $\beta(t)=x$.
Since $u$ is not a cut-vertex, $x$ has a neighbor $w\neq u$ in $A$, joined by an edge $g$.
If an immersion $\beta$ with $\beta(s)=w$ and $\beta(st)=g$ did not exist, then by the min-cut-max-flow duality, there
would exist a separation $(C,D)$ of $G[A]-g$ with $u\in D$ of order at most $2-2[w\in D]-[x\in D]$.
Then $(C,D\cup B)$ would be a separation of $G$ of order at most $2-2[w\in D]-[x\in D]+[x\in C]+[|\{w,x\}\cap D|=1]$.
Since $G$ is $3$-edge-connected, we would have $x,w\in C$; but then $(C,D\cup B)$ would be a separation of $G$ of order $3$ with
both sides of size at least two, contradicting the assumption that $G$ is internally $4$-edge-connected.
Therefore, the immersion $\beta$ as described exists.

By symmetry, there also exists an immersion $\beta'$ of the triangle $rst$ with double edge $rs$ in $G[A]$ such that $\beta'(r)=v$ and $\beta'(t)=y$.
But then $\beta$, $\beta'$, and the edges of $\partial(A,B)$ produce an immersion of $K_{3,3}$ in $G$, which is a contradiction.
\end{proof}

The following two results are proved by computer-assisted enumeration.   The program used to verify the claims can be found
as part of the arXiv submission.

\begin{lemma}\label{lemma-no9}
There are no graphs in $\pec\setminus\subp$ with exactly $9$ vertices.
\end{lemma}
\begin{proof}
For contradiction, suppose this is not the case, and thus there exists a graph $G\in\underline{\pec}$ with $9$ vertices.
By Lemmas~\ref{lemma-nonpl} and \ref{lemma-2con}, $G$ is planar and $2$-connected, and the multiplicities of its edges are at most three by Lemma~\ref{lemma-maxmul}.
Using \texttt{plantri}, we enumerated all planar $2$-connected simple graphs with $9$ vertices.
Via a computer program, we tested all possible minimal ways how to increase the multiplicities of their edges to at most three so
that the resulting graph is weakly $5$-edge-connected and not $3$-regular, and verified that all graphs obtained in this way
contain an immersion of $K_{3,3}$.  As $G\in\underline{\pec}$, $G$ would have to appear among these graphs, which is a contradiction.
\end{proof}

\begin{figure}
\begin{center}
\includegraphics{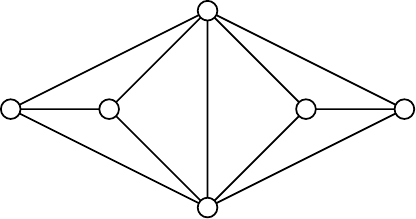}
%\begin{asy}
%unitsize(10mm);
%
%void vertex(pair a, pen barva=white, real pol = 0.1)
%{
%  filldraw(circle (a, pol), fillpen=barva);
%}
%
%pair v[];
%int i;
%v[0]=(0,0);
%v[1]=(0,2);
%v[2]=(-2,1);
%v[3]=(-1,1);
%v[4]=(1,1);
%v[5]=(2,1);
%
%draw(v[0]--v[2]--v[3]--v[1]--v[4]--v[5]--v[0]--v[1]--v[5]);
%draw(v[1]--v[2]);
%draw(v[3]--v[0]--v[4]);
%
%for (i = 0; i < 6; ++i)
%  vertex(v[i]);
%\end{asy}
\end{center}
\caption{An exceptional graph from Lemma~\ref{lemma-no678}.}\label{fig-except}
\end{figure}

A similar computer-assisted argument gives us the following.
\begin{lemma}\label{lemma-no678}
The graph depicted in Figure~\ref{fig-except} is the only graph $G\in\pec\setminus\subp$ with $6$, $7$, or $8$ vertices
such that $G$ is $2$-connected, does not contain edges of multiplicity greater than four, does not contain vertices
of degree exactly four, and some vertex of $G$ of degree exactly five is adjacent to a vertex of degree at least five.
\end{lemma}

We now consider the properties of (hypothetical) graphs in $G\in\underline{\pec}$ with at least $10$ vertices.
Firstly, these graphs have very limited multiplicities of edges.

\begin{lemma}\label{lemma-pare}
Let $n\ge 10$ be an integer such that $\pec\setminus\subp$ contains no graphs with exactly $n-1$ vertices,
and let $G\in\underline{\pec}$ have exactly $n$ vertices.  No vertex $v$ of $G$ is incident with an edge of multiplicity at least $\deg(v)/2$,
and in particular, each vertex of $G$ has at least three distinct neighbors.
\end{lemma}
\begin{proof}
Suppose for a contradiction that a vertex $v$ of degree $d$ is adjacent to $w$ via an edge of multiplicity $m\ge d/2$.
Let $G'$ be the graph obtained from $G$ by contracting $v$ and $w$ to a new vertex $x$.  Clearly, $G'$ is weakly $5$-edge-connected.
Furthermore, $G'$ can be seen as created from $G$ by splitting off pairs of edges at $v$ and then deleting $v$,
and thus $G'$ does not contain $K_{3,3}$ as a minor.  Furthermore, $\deg_{G'}(x)=|\partial_G(\{v,w\})|\ge 4$, since $G$ is internally $4$-edge-connected,
and thus $G'$ is not $3$-regular.  But $G'$ has exactly $n-1$ vertices, contradicting the assumptions of the lemma.
\end{proof}

Secondly, these graphs have no vertices of degree exactly four.

\begin{lemma}\label{lemma-nodeg4}
Let $n\ge 10$ be an integer such that $\pec\setminus\subp$ contains no graphs with exactly $n-1$ vertices,
and let $G\in\underline{\pec}$ have exactly $n$ vertices.  No vertex $v$ of $G$ has degree exactly four.
\end{lemma}
\begin{proof}
Suppose for a contradiction that $\deg(v)=4$.  By Lemma~\ref{lemma-pare}, $v$ is only incident with single edges.
Let $e_1=vv_1$, \ldots, $e_4=vv_4$ be the edges incident with $v$ in the cyclic ordering according to the drawing of $G$.
Let $G'$ be the graph obtained from $G$ by splitting off $e_1$ with $e_3$ and $e_2$ with $e_4$, and deleting $v$.
Clearly, $G'$ does not contain $K_{3,3}$ as an immersion.  By Lemma~\ref{lemma-3nbr}, at least one of $v_1$, \ldots, $v_4$
has degree at least four, and thus $G'$ is not $3$-regular.  Since $G'$ has exactly $n-1$ vertices and
no such graph belongs to $\pec\setminus\subp$, we conclude that $G'$ is not weakly $5$-edge-connected.

Hence, $G'$ has a separation $(A,B)$ without an acceptable side.  If $|\{v_1,v_3\}\cap A|=1$ and $|\{v_2,v_4\}|\cap B=1$,
then note that $(A\cup \{v\},B)$ is a separation of $G$ of order $|\partial_{G'}(A,B)|\le 4$, and since $G$ is weakly $5$-edge-connected,
we conclude $|B|=2$.  However, the same argument applied to $(A,B\cup\{v\})$ shows $|A|=2$, contradicting the assumption that $G$ has $n\ge 10$
vertices.   Therefore, we can by symmetry assume $\{v_1,v_3\}\subseteq A$.  If $|\{v_2,v_4\}\cap A|\ge 1$,
then $(A\cup \{v\},B)$ is a separation of $G$ of order $|\partial_{G'}(A,B)|\le 4$.  Since $G$ is weakly $5$-edge-connected, either
the order of the separations is $3$ and $|B|=1$, or the order is $4$ and $|B|\le 2$; in either case, $B$ is an acceptable side of
$(A,B)$ in $G'$, which is a contradiction.  Therefore, $\{v_2,v_4\}\subseteq B$.

By planarity, it cannot be the case that both $G[A]$ contains a path from $v_1$ to $v_3$ and $G[B]$ contains a path from $v_2$ to $v_4$.
By symmetry, we can assume $v_1$ and $v_3$ are in different components $A_1$ and $A_2$ of $G[A]$.  However,
$4\ge \partial_{G'}(A)=|\partial_G(A_1)|+|\partial_G(A_2)|-2\ge 4$, since $G$ is $3$-edge-connected.  This implies
$|\partial_G(A_1)|=|\partial_G(A_2)|=3$, and thus $|A_1|=|A_2|=1$, since $G$ is internally $4$-edge-connected.
But then $|\partial_{G'}(A)|=4$ and $|A|=2$, and thus $A$ is an acceptable side of $(A,B)$ in $G'$, which is a contradiction.
\end{proof}

Finally, we are ready to prove the main result.

\begin{theorem}
Every graph in $\pec\setminus\subp$ has at most $8$ vertices.
\end{theorem}
\begin{proof}
Suppose for a contradiction that $n\ge 9$ is the smallest integer such that $\pec\setminus\subp$ contains a graph $G$ with $n$ vertices.
We can assume that $G\in \underline{\pec}$.  By Lemma~\ref{lemma-no9}, we have $n\ge 10$.  By Lemma~\ref{lemma-2con}, $G$ is $2$-connected.
By Lemma~\ref{lemma-nodeg4}, $G$ does not contain a vertex of degree exactly $4$, and since $G$ is not $3$-regular, it contains a vertex $v$ of
degree at least five.  By Lemma~\ref{lemma-pare}, all edges of $G$ incident with a vertex of degree three are simple, and thus by Lemma~\ref{lemma-3nbr},
$v$ has a neighbor $u$ of degree greater than three; by Lemma~\ref{lemma-nodeg4}, $u$ also has degree at least five.  Let $e$ be an edge joining $u$ with $v$.

Observe that $G$ does not contain a separation $(A,B)$ of order at most $4$ such that $v\in A$ and $|B|\ge 3$.  Indeed, since $\deg(v)\ge 5$,
we would have $|A|\ge 2$, and since $G$ is weakly $5$-edge-connected, we would have $|A|=2$.  Letting $A=\{v,x\}$ and denoting by $m$
the multiplicity of the edge between $v$ and $x$, Lemma~\ref{lemma-pare} would imply
$4\ge |\partial_G(A,B)|=\deg_G(v)+\deg_G(x)-2m\ge \deg_G(v)$, which is a contradiction.
Symmetrically, $G$ does not contain a separation $(A,B)$ of order at most $4$ such that $u\in B$ and $|A|\ge 3$.

The graph $G-e$ is not $3$-regular and does not contain $K_{3,3}$ as an immersion, and since $G\in \underline{\pec}$, it follows $G-e$ is not
weakly $5$-edge-connected.  Hence, $G-e$ has a separation $(A,B)$ without an acceptable side, where say $v\in A$ and $u\in B$.
By the previous paragraph, $|\partial_G(A,B)|=5$, and since neither $A$ nor $B$ is acceptable in $G-e$, we have $|A|,|B|\ge 3$.

Let $G_A$ be the graph obtained from $G$ by contracting $A$ to a single vertex $a$ of degree five.  Note that by Menger's theorem,
$G$ contains five pairwise edge-disjoint paths from $v$ to $B$, and thus $G_A$ is immersed in $G$; hence, $G_A$ does not contain $K_{3,3}$
as an immersion.  Furthermore, $G_A$ is clearly weakly $5$-edge-connected and not $4$-regular, and thus $G_A\in\pec\setminus\subp$.
Since $|V(G_A)|<n$, the choice of $n$ implies that $|V(G_A)|\le 8$.  Since $G[B]$ is connected and $G$ is $2$-connected, $G_A$ is $2$-connected.
By Lemma~\ref{lemma-maxmul}, no edge of $G$ has multiplicity greater than three, and since $G$ is $2$-connected, the contraction of $A$
could not result in an edge of multiplicity five, and thus $G_A$ has no edge of multiplicity greater than four.  Furthermore, $G_A$
does not contain vertices of degree four, and contains a vertex of degree exactly five adjacent to a vertex of degree at least five.
By Lemma~\ref{lemma-no678}, either $G_A$ is the graph depicted in Figure~\ref{fig-except}, or $|V(G_A)|\le 5$.

By symmetry, the same claim holds for the graph $G_B$ obtained from $G$ by contracting $B$ to a single vertex.  It cannot be the case
that both $G_A$ and $G_B$ are isomorphic to the graph depicted in Figure~\ref{fig-except}, as otherwise $G$ would contain a $4$-cycle
of vertices of degree three, contradicting the assumption that $G$ is weakly $5$-edge-connected.  Hence, we can by symmetry assume
$|V(G_A)|\le 5$ and $|V(G_B)|\le 6$.  However, then $|V(G)|=|V(G_A)|+|V(G_B)|-2\le 9$.  This contradicts Lemma~\ref{lemma-no9}.
\end{proof}

\section*{Acknowledgments}

We would like to thank Tom\'a\v{s} Kaiser, Tereza Klimo\v{s}ov\'a and Daniel Kr\'al' for useful discussions at the beginning of the project.

\bibliographystyle{siam}
\bibliography{../data.bib}

\begin{thebibliography}{1}

\bibitem{devos2018structure}
{\sc M.~DeVos and M.~Malekian}, {\em The structure of graphs with no
  {$K_{3,3}$} immersion}, arXiv, 1810.12873 (2018).

\bibitem{devos2018structurew4}
\leavevmode\vrule height 2pt depth -1.6pt width 23pt, {\em The structure of
  graphs with no {$W_4$} immersion}, arXiv, 1810.12863 (2018).

\bibitem{giannopoulou2015forbidding}
{\sc A.~C. Giannopoulou, M.~Kami{\'n}ski, and D.~M. Thilikos}, {\em Forbidding
  kuratowski graphs as immersions}, Journal of Graph Theory, 78 (2015),
  pp.~43--60.

\bibitem{gmarx}
{\sc M.~Grohe and D.~Marx}, {\em Structure theorem and isomorphism test for
  graphs with excluded topological subgraphs}, SIAM J. Comput., 44 (2015),
  pp.~114--159.

\bibitem{robertson2003graph}
{\sc N.~Robertson and P.~D. Seymour}, {\em {Graph {M}inors. {XVI}. {E}xcluding
  a non-planar graph}}, J. Combin. Theory, Ser.~B, 89 (2003), pp.~43--76.

\bibitem{wagner}
{\sc K.~Wagner}, {\em {\"U}ber eine {E}igenschaft der ebenen {K}omplexe},
  Mathematische Annalen, 114 (1937), pp.~570--590.

\bibitem{wollims}
{\sc P.~Wollan}, {\em The structure of graphs not admitting a fixed immersion},
  Journal of Combinatorial Theory, Series B, 110 (2015), pp.~47--66.

\end{thebibliography}

\end{document}